\documentclass[12pt]{article}  % `article' with A4 paper size option
\usepackage{amsmath,amssymb,color,amscd}

\def\seq#1#2#3{#1_{#2},\,\ldots,#1_{#3}}

\def\h{\widehat}

\def\vv{{\underline{v}}}

\def\tt{{\underline{t}}}

\def\ww{\underline{w}}

\def\1{\underline{1}}

\def\P{\mathbb P}

\def\Z{\mathbb Z}

\def\C{\mathbb C}

\def\CP{\mathbb C\mathbb P}
\def\OO{{\cal O}}

\newtheorem{theorem}{Theorem}
\newtheorem{statement}{Statement}

\newenvironment{definition}
{\smallskip\noindent{\bf Definition\/}:}{\smallskip\par}
\newenvironment{example}
{\smallskip\noindent{\bf Example\/}.}{\smallskip\par}
\newenvironment{examples}
{\smallskip\noindent{\bf Examples\/}.}{\medskip\par}
\newenvironment{remark}
{\smallskip\noindent{\bf Remark\/}.}{\smallskip\par}
\newenvironment{remarks}
{\smallskip\noindent{\bf Remarks\/}.}{\smallskip\par}
\newenvironment{proof}
{\noindent{\bf Proof\/}.}{{ $\Box$}\smallskip\par}

\title{Equivariant Poincar\'e series of filtrations and topology}

\author{
A.~Campillo
\and F.~Delgado \and S.M.~Gusein-Zade
\thanks{
Math. Subject Class. 14B05, 16W70,
16W22. Keywords: finite group actions,
filtrations, Poincar\'e series.
Partially supported by the grant MTM2007-64704 (with the help of
FEDER Program). Third author is also partially supported by the
Russian government grant 11.G34.31.0005, RFBR--10-01-00678,
NSh--4850.2012.1 and Simons-IUM fellowship.
} }

\date{}
\begin{document}
%\sloppy
\def\eps{\varepsilon}

\maketitle

\begin{abstract}
Earlier, for an action of a finite group $G$ on a germ of an analytic variety, an equivariant $G$-Poincar\'e series of a multi-index
filtration in the ring of germs of functions on the variety was defined as an element of the Gro\-then\-dieck ring of $G$-sets with an additional structure.
We discuss to which extend the $G$-Poincar\'e series of a filtration defined by a set of curve
or divisorial valuations
on the ring of germs of analytic functions in two variables determines the (equivariant) topology of the curve or of the set of divisors.
\end{abstract}

%%%%%%%%%%%%%%%%%%%%%%%%%%%%%%%%%%%%%%%%%%%%%%%%%%%%%%%%%%%%%%%%%
\section*{Introduction}\label{sec0}
%%%%%%%%%%%%%%%%%%%%%%%%%%%%%%%%%%%%%%%%%%%%%%%%%%%%%%%%%%%%%%%%%
The Poincar\'e series of a multi-index filtration (say, on the ring
of germs of functions on a variety) was defined in \cite{CDK}. It
was computed for filtrations
on the ring $\OO_{\C^2,0}$ of germs of analytic functions in two variables
corresponding to plane curve
singularities with several branches \cite{duke} and for divisorial
ones \cite{divisorial}. In \cite{duke} it was found that the
Poincar\'e series of the filtration defined by a plane curve
singularity $(C,0)\subset(\C^2,0)$ coincides with the Alexander
polynomial in several variables of the corresponding link $C\cap
S^3_{\eps}\subset S^3_{\eps}$ ($S^3_{\eps}$ is the sphere of a small
radius $\eps$ centred at the origin in $\C^2$). Therefore it defines
the (embedded) topology of the plane  curve singularity
\cite{yamamoto}. Identifying all the variables in the Alexander
polynomial one gets the monodromy  zeta function of the singularity.
In \cite{FAOM} it was shown  that the Poincar\'e series of a
divisorial filtration in the ring $\OO_{\C^2,0}$ of germs of
functions in two variables also defines the topology of the
corresponding set of divisors (more precisely the topology of its
minimal resolution). The corresponding statement for the divisorial
filtration defined by {\em all} components of the exceptional
divisor of a resolution of a normal surface
singularity was obtained in \cite{cutkosky}.

The intention to generalize connections between Poincar\'e series of
filtrations and monodromy zeta functions to an equivariant context
led to the desire to define equivariant analogues of Poincar\'e
series and of monodromy zeta functions. In particular, in
\cite{RMC} there was defined an equivariant analogue of the
Poincar\'e series of a multi-index filtration on the ring of
functions on a complex analytic space singularity with an action of
a finite group $G$. This $G$-Poincar\'e series is an element of the
Grothendieck ring $K_0((G,r)\mbox{-sets})$ of $G$-sets with an
additional structure. (For the trivial group $G$ this ring coincides
with the ring $\Z[[t_1, \ldots, t_r]]$ of power series in several
variables.) It was computed for the filtrations defined by plane
curve singularities and for divisorial filtrations in the plane (see Section~\ref{defPoin} below).
Here we discuss to which extend the $G$-Poincar\'e series of these
filtrations determine the (equivariant) topology of the curve or of the set of
divisors.
We show that the $G$-Poincar\'e series of a collection of divisorial valuations determines the topology of the set of divisors. This is not, in general, the case for curve valuations. We describe some conditions on curves under which the corresponding statement holds. It remains unclear whether the (equivariant) topology of a collection of curves always determines the $G$-Poincar\'e series of the collection.

%%%%%%%%%%%%%%%%%%%%%%%%%%%%%%%%%%%%%%%%%%%%%%%%%%%%%
\section{$G$-equivariant resolutions}\label{sec1}
%%%%%%%%%%%%%%%%%%%%%%%%%%%%%%%%%%%%%%%%%%%%%%%%%%%%%

%% \section{$G$-equivariant resolutions of plane curve singularities and of sets of
%divisors}\label{sec1}

It is well known that two plane curve singularities are
topologically equivalent if and only if they have combinatorially
equivalent (embedded) resolutions: see e.g. \cite{Wall}. The formula for the Poincar\'e series of a plane curve
singularity in terms of a resolution from \cite{duke} implies, in particular, that
the Poincar\'e series of a plane curve singularity is a topological invariant. The notion of topological equivalence
of two sets of divisors in \cite{FAOM} was in fact formulated in
terms of topologically equivalent resolutions.

\begin{remark}
A divisorial valuation can be defined by a (generic) pair of curvettes corresponding to the divisor (see below).
This way a set of divisors can be defined by a curve (the union of the corresponding pairs of curvettes) and two
sets of divisors are topologically equivalent if and only if the corresponding curves are topologically equivalent.
\end{remark}

Here we discuss the concept of topologically equivalent resolutions in an
equivariant setting.

Let a finite group $G$ act on $(\C^2,0)$ (by complex analytic transformations).
Without loss of generality one can assume that this action is faithful and is
defined by a two-dimensional representation of the group $G$ (i.e. that $G$ acts on $\C^2$ by
linear transformations).

Let $(C_i,0)\subset(\C^2,0)$, $i=1, \ldots, r$, be (different) irreducible plane curve
singularities: branches.

\begin{remark}
Here we do not assume, in general, that the curve $\bigcup_{i=1}^r C_i$ is $G$-invariant (i.e.
that the set $\{C_i\}$ contains all $G$-shifts of its elements) or that all the branches $C_i$ belong to
different orbits of the $G$-action. Restrictions of this sort could be required for
particular statements.
\end{remark}

\begin{definition}
A {\em $G$-equivariant resolution} (or simply a {\em
$G$-resolution}) of the set $\{C_i\}$ (or of the curve $\bigcup_{i=1}^r C_i$) is a proper complex
analytic map $\pi: ({\cal{X, D}})\to (\C^2,0)$ from a
smooth surface ${\cal X}$ with an action of the group $G$ such that:
\begin{enumerate}
\item[1)]
$\pi$ is an isomorphism outside of the origin in $\C^2$;
\item[2)]
$\pi$ commutes with the $G$-actions on ${\cal X}$ and on $\C^2$;
\item[3)]
the total transform $\pi^{-1}(\bigcup_{i=1}^r C_i)$ of the curve $\bigcup_{i=1}^r C_i$ is a normal
crossing divisor on ${\cal X}$;
\item[4)]
for each branch $C_i$ its strict transform
$\widetilde{C}_i$ is a germ of a smooth curve transversal to the
exceptional divisor ${\cal{D}}=\pi^{-1}(0)$ at a smooth point $x$ of
it and is invariant with respect to the isotropy subgroup
$G_x=\{g\in G: gx=x\}$ of the point $x$.
\end{enumerate}
\end{definition}

\begin{remarks}
{\bf 1.} The resolution $\pi$ can be obtained by a sequence of
blow-ups of points (preimages of the origin). The exceptional
divisor $\cal{D}$ is the union of its irreducible components
$E_{\sigma}$, $\sigma\in\Gamma$. The set $\Gamma$ inherits the
partial order defined by a representation of $\pi$ as a sequence of
blow-ups: a component $E_{\sigma'}$ is greater than another
component $E_{\sigma}$ ($\sigma'>\sigma$) if the exceptional divisor
of any modification which contains $E_{\sigma'}$ also
contains $E_{\sigma}$. The condition 2) means that this sequence of
blow-ups should be $G$-equivariant, i.e., if a point $x$ is blown-up,
the point $gx$ should be blown-up for each $g\in G$ as well. In
particular, the set $\Gamma$ of the components of the exceptional
divisor is a $G$-set.

{\bf 2.} The condition 4) is equivalent to say that $\pi$ is a resolution of the
curve $C=\bigcup\limits_{i,g}gC_i$ where $g$ runs through all the elements of $G$, $i=1,\ldots,r$. In
particular, $\pi^{-1}(C)$ is a normal crossing divisor on ${\cal X}$.
A smooth irreducible curve $(C_1,0)\subset(\C^2,0)$ has a trivial
$G$-resolution $(\C^2,0){\stackrel{=}{\to}}(\C^2,0)$ if and only if it is
$G$-invariant.
\end{remarks}

The group $G$ acts both on the space ${\cal X}$ of the resolution
and on the exceptional divisor ${\cal D}$. For $\sigma\in\Gamma$,
let $G_{\sigma}$ be the isotropy subgroup of the component
$E_{\sigma}$, i.e. $\{g\in G: g E_{\sigma}= E_{\sigma}\}$. Pay
attention that the isotropy subgroups $G_{\sigma}$ are Abelian for all
$\sigma$ except possibly the first (minimal) one. (In the latter case
$G_{\sigma}$ coincides with the group $G$ itself.) The group
$G_{\sigma}$ acts on the component $E_{\sigma}$. Let
$G_{\sigma}^*\subset G_{\sigma}$ be the isotropy subgroup
$G_x=\{g\in G: gx=x\}$ of a generic point $x\in E_{\sigma}$. (The
group $G_{\sigma}^*$ is always Abelian.) A point $x\in E_{\sigma}$
will be called {\em special} if its isotropy subgroup $G_x$ is
different from $G_{\sigma}^*$. (In this case
$G_x\supset G_{\sigma}^*, G_x\neq G_{\sigma}^*$.) One has a one dimensional representation
$\beta_{x\sigma}$ of the isotropy subgroup $G_x$ of a point
$x\in E_{\sigma}$ in the normal space to $E_{\sigma}$ in $\cal{X}$
at the point $x$.

Let $E_\sigma$ be a component of the exceptional divisor ${\cal D}$, let $x\in
E_\sigma$ be a smooth
point of ${\cal D}$, i.e. a point which is not a point of intersection with other
components of $\{\cal D\}$,  and let $\widetilde L$ be a germ
of a smooth curve on ${\cal X}$ at the point $x$ invariant with respect to the
isotropy group $G_x$ of the point $x$. The curve $L=\pi(\widetilde{L})$ is called a
curvette corresponding to the component $E_{\sigma}$ (and/or to the point $x$).

A $G$-resolution of a curve $(\bigcup_{i=1}^r C_i,0)\subset(\C^2,0)$ can be described
by its dual resolution graph in the following way. The vertices of
this graph correspond to the components $E_{\sigma}$ of the
exceptional divisor $\cal{D}$ (i.e. to the elements of the partially ordered $G$-set
$\Gamma$) and to the strict transforms of the the curves $C_i$
and of their $G$-shifts $gC_i$, $g\in G$.
These vertices should be depicted by bullets and arrows respectively.
There is a natural $G$-action on the set of vertices of the graph
(preserving bullets and arrows).
Two vertices are
connected by an edge if and only if the corresponding components of
the total transform $\pi^{-1}(C)$ of the curve $C=\bigcup\limits_{i,g}gC_i$ intersect.

\begin{remark}
%% {\bf 1.}
One should have in mind that the described information
(namely the $G$-action on the set of vertices of the graph)
determines the isotropy subgroups $G_{\sigma}$ of the components
and the isotropy subgroups of all the intersection points of
the components of the total transform $\pi^{-1}(C)$ of the curve
$C$ (all the latter subgroups are Abelian).
\end{remark}

The same definition and description apply to a set of divisorial
valuations with the  only difference that in this case the
corresponding divisors should be indicated and there are no strict
transforms of branches. Also Remark~1 above is valid.

In what follows we shall use the following description of the
behaviour of representations under blow-ups. Assume than one has
the complex plane $\C^2$ with a representation of a finite Abelian
group $H$. This representation is the sum of two irreducible
ones, say $\gamma_1$ and $\gamma_2$. Let
$p: (Y,E)\to (\C^2,0)$ be the blow-up of the origin ($E\simeq
\CP^1$). The group $H$ acts on $Y$ as well.

If the
representations $\gamma_1$ and $\gamma_2$ are different then the
$H$-action on $E$ has two special points invariant with respect to $H$.
At one of them the
representation of $H$ in the normal space to $E$ is $\gamma_1$
and in the tangent space to $E$ is $\gamma_2 \gamma_1^{-1}$. At
the other one they are $\gamma_2$ and $\gamma_1 \gamma_2^{-1}$
respectively. The isotropy subgroup $H_x$  of a non-special point
$x$ of $E$ is $H_x = \{h\in H : \gamma_1(h)=\gamma_2(h)\}$. The
representation of $H_x$ in the normal space to $E$ is
$\gamma_1|_{H_x}=\gamma_2|_{H_x}$. (The representation of $H_x$
in the tangent space to $E$ is trivial.)

If the representations $\gamma_1$ and $\gamma_2$ coincide, there
are no special points on $E$, the action of $H$ on $E$ is trivial
and the representation of $H_x=H$ in the normal space to $E$ is
$\gamma_1=\gamma_2$.

Thus the representation of $H$ on $\C^2$ determines the
representations in the tangent and in the normal spaces to $E$ at
all points.

Let $\{C_i\}$ and $\{C'_i\}$, $i=1,\ldots, r$, be two collections of branches in the complex plane
$(\C^2, 0)$ with an action of
a finite group $G$. We say that these collections are $G$-topologically equivalent
if there exists a $G$-invariant
germ of a homeomorphism $\psi:(\C^2, 0)\to(\C^2, 0)$ such that $\psi(C_i)=C'_i$ for $i=1,\ldots,r$. A
version of this definition can be applied to collections of divisorial valuations
as well. A divisorial valuation $v$ on $\OO_{\C^2,0}$ can be described by a
generic pair of curvettes corresponding to the divisor. (Genericity means that the
strict transforms of the curvettes intersect the divisor at different points.)
Two collections of divisorial valuations $\{v_i\}$ and $\{v'_i\}$, $i=1,\ldots,r$,
are said to be $G$-topologically equivalent if the corresponding collections of curvettes
$\{L_{ij}\}$ and $\{L'_{ij}\}$, $i=1,\ldots,r$, $j=1,2$, are $G$-topologically equivalent.

It is clear that the $G$-resolution graph of a collection of curve or divisorial
valuations does not determine the $G$-equivariant topology of the collection.
Moreover the $G$-resolution graphs of collections can be the same for different
actions of the group $G$ on $\C^2$ (e.g. if $G$ is abelian and all blow-ups are
performed at points with the isotropy subgroups $G_x = G$). Even if the representation of
$G$ is fixed, the $G$-resolution graph of a collection of curve or divisorial
valuations does not determine the $G$-topology of the collection.

\begin{examples}
{\bfseries 1.}
Let $\C^2$ be the complex plane with the action of the cyclic group $G=\Z_{15}$
defined by $\sigma * (x, y)=({\sigma}^3x, {\sigma}^{5}y)$, where
$\sigma=\exp({2\pi i}/{15})$ is the generator of the group $\Z_{15}$. Let
$C_i$, $i=1,2,3$, be the curves (given by their parameterizations) $(t,0)$,
$(0,t)$, $(t,t^2)$, respectively and let $C'_i$, $i=1,2,3$, be
the curves $(0,t)$, $(t, 0)$, $(t^2,t)$. An important property of
these curves is that no element of $G$ different from $1$ sends the curve
$C_3$ (or the curve $C'_3$) to itself and that, in
the minimal
$G$-resolution, the
strict transform of the curve $C_3$ (or of the curve $C'_3$) intersects
a component
$E_{\sigma}$ of
the exceptional divisor with $G_{\sigma}=G$ and $G^*_{\sigma}=(e)$. (Thus the strict
transforms of all the $G$-shifts of the curve intersect one and the same
component of the exceptional divisor at different points.)

The minimal $G$-resolution graph is shown on Figure 1 (it is one and the same for both
cases). The partial order on $\Gamma$ is defined by the numbering of the elements of $\Gamma$ (vertices).
The action of $G$ on the set $E_{\sigma}$ of components of the exceptional
divisor is trivial; the curves with the numbers 1 and 2 are $G$-invariant and all the
$G$-shifts of the third curve are different.

%%%%%%%%%%%%
%Fig 1
\begin{figure}[h]
$$
\unitlength=1.00mm
\begin{picture}(90.00,35.00)(0,0)
\thicklines \put(20,10){\line(1,0){40}} \put(20,10){\circle*{2}}
%\put(40,10){\circle*{2}}
\put(60,10){\circle*{2}}
\put(20,10){\vector(0,1){20}} \put(60,10){\vector(0,1){20}}
\put(60,10){\vector(2,1){20}} \put(60,10){\vector(2,-1){20}}
%\put(60,10){\vector(5,1){20}} \put(60,10){\vector(5,-1){20}}
\put(75,16){\circle*{0.2}}
\put(75,14){\circle*{0.2}}
\put(75,12){\circle*{0.2}}
\put(75,10){\circle*{0.2}}
\put(75,8){\circle*{0.2}}
\put(75,6){\circle*{0.2}}
\put(75,4){\circle*{0.2}}
\put(16,5){$1$} \put(56,5){$2$}
\put(22,25){$C_2$} \put(62,25){$C_1$} \put(82,20){$C_3$}
\put(82,0){$gC_3$}
\end{picture}
$$
\caption{Example 1}
\label{fig1}
\end{figure}
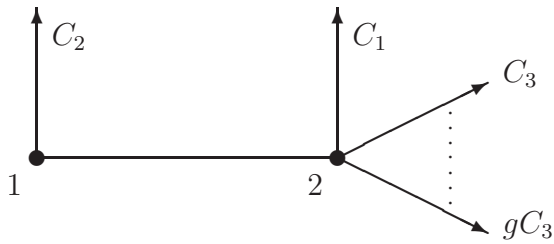
%%%%%%%%%%%

A local $G$-equivariant homeomorphism
$(\C^2,0)\to (\C^2,0)$ which sends the curve $C_i$ to the curve $C'_i$ does not
exists since the isotropy groups of the curves
$C_1\setminus\{0\}$ and
$C'_1\setminus\{0\} = C_2\setminus \{0\}$ are different.
Moreover there is no
homeomorphism
$(\C^2,0)\to (\C^2,0)$ which sends $C_3$ to $C'_3$, $C_1$ to $C'_2$ and $C_2$ to
$C'_1$ because a homeomorphism has to preserve the intersection multiplicities of
branches.

{\bfseries 2.}
The fact that the action of the group $G=\Z_{15}$ on $\C^2\setminus\{0\}$ has
different isotropy subgroups of different points is not really essential. Let
$G = \Z_7$ with the action $\sigma*(x,y)= (\sigma x, \sigma^3 y)$, where
$\sigma = \exp(2 \pi i/7)$. Let $C_i$ and $C'_i$ ($i=1,2,3$) be defined as in
Example 1. The same arguments (based on the intersection multiplicities) implies that a
local homeomorphism
$(\C^2,0)\to (\C^2,0)$ which sends $C_3$ to $C'_3$ should send $C_1$ to $C'_1=C_2$.
However there is no $G$-equivariant local homeomorphism from $(\C,0)$ with the
$G$-action $\sigma*z = \sigma z$ to $(\C,0)$ with the $G$-action
$\sigma*z = \sigma^3 z$.

The argument in Example 1 can be easily adapted to the case of divisorial valuations.

{\bfseries 3.}
Let $G$ be the group $\Z_{15}$ with
the same action on $(\C^2,0)$ as in Example 1. The
divisorial valuation $v$ (resp. $v'$) is defined by the following two curvettes:
$C_1:= \{(t,t^2)\}$ and $C_2:=\{(t,-t^2)\}$
(resp. $C'_1:= \{(t^2,t)\}$ and $C'_2:=\{(-t^2,t)\}$). The minimal resolution graph
for both cases is shown on Figure 2. The component of the exceptional divisor corresponding to the valuation is distinguished (marked by the circle). The $G$-action on it is trivial.

%%%%%%%%%%%%%%%%%%5
%Figure 2
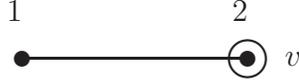
\begin{figure}[h]
$$
\unitlength=1.00mm
\begin{picture}(50.00,10.00)(0,0)
\thicklines
\put(10,2){\line(1,0){30}}
\put(10,2){\circle*{2}}
\put(40,2){\circle*{2}}
\put(40,2){\circle{5}}
\put(8,7){$1$} \put(38,7){$2$}
\put(45,1){$v$}
\end{picture}
$$
\caption{Example 3}
\label{fig2}
\end{figure}
%%%%%%%%%%%%%%%%%%%%%

A local $G$-equivariant homeomorphism
$(\C^2,0)\to (\C^2,0)$ should preserve the $x$ and the $y$-axis (because of the
isotropy subgroups). However such homeomorphism can not send a curvette corresponding to
$v$ to a curvette corresponding to $v'$ since it should preserve the intersection
multiplicities of branches (in other words, since $v(x)\neq v'(x)$).
\end{examples}

For an Abelian group $G$, special points on the first component $E_1$
of the exceptional
divisor (obtained by blowing-up the origin in $\C^2$) exist if and only if the
representation of $G$ is the sum of two different one-dimensional representations of
$G$ and they correspond to these representations. If the group $G$ is not Abelian, special
points on $E_1$ correspond to Abelian subgroups of $G$ and their
one-dimensional representations.

\begin{theorem}
Assume that the initial action (representation) of the group $G$ on $\C^2$ is fixed.
Then a $G$-resolution graph of a collection of curve or divisorial valuations with
the correspondence between the ``tails" of the graph (i.e. the connected components
of the graph without the first vertex 1) and the special points on $E_1$ determines
the $G$-topology of the collection of the curves or of the divisorial valuations.
\end{theorem}

%% \begin{remark}
%% Moreover one can see that the data described in Theorem 1 determine the minimal
%% $G$-resolution graph of the collection.
%% \end{remark}

\begin{proof}
The case of divisorial valuations is formulated in terms of curves (via pairs of corresponding curvettes) and thus follows from the curve case. We shall show that if $G$-resolutions of the collections of curves $\{C_i\}$ and $\{C'_i\}$, $i=1, \ldots, r$, are described by the same data (i.e. they lie in the same $\C^2$ with a $G$-representation, have the same $G$-resolution graphs and the same correspondences between the tails of the graphs and the special points on $E_1$), then there exists a $G$-equivariant homeomorphism (in fact a $C^{\infty}$-diffeomorphism) $\psi:({\cal{X, D}})\to({\cal{X', D'}})$ of a neighbourhood of the exceptional divisor $\cal D$ in $\cal X$ to a neighbourhood of $\cal D'$ in $\cal X'$ sending the strict transforms $\widetilde{C}_i$ of the curves $C_i$ to the strict transforms $\widetilde{C}'_i$ of the curves $C'_i$. Blowing down this diffeomorphism one obtains the required homeomorphism $(\C^2,0)\to(\C^2,0)$.

Such diffeomorphism can be constructed inductively following the processes of the $G$-resolutions. We shall show that after each blow-up (or rather after a set of blow-ups at the points from one $G$-orbit) one can construct a $G$-equivariant diffeomorphism of a neighbourhood of the new exceptional divisor(s) in the first resolution to a neighbourhood of the one(s) in the second resolution which sends intersection points of the strict transforms of the branches $C_i$ with the exceptional divisor to the corresponding points for the branches $C'_i$. Moreover we shall show that this can be made in such a way that the diffeomorphism remains complex analytic in neighborhoods of all these intersection points.

After the first blow-up one has an identification (an analytic one) of the first exceptional divisors and of their neighbourhoods in the both resolutions. This identification keeps special points. There can be some intersection points of the strict transforms of the branches $C_i$ with the exceptional divisor which are not special ones and the corresponding points for the branches $C'_i$. The diffeomorphism has to send first points to the latter ones. This can be made by a (smooth) isotopy of the initial diffeomorphism (the identification in this case), say, with the help of a $G$-invariant vector field which brings images of the points to the required ones. Moreover this can be made so that the diffeomorphism remains complex analytic in neighbourhoods of these points.

The description of the behaviour of a representation under the blow-up shows that, for each point $x$ in the exceptional divisor of the first resolution, its isotropy subgroup $G_x$ and its representations in the tangent and in the normal spaces to the exceptional divisor coincide with those for the corresponding image $x'=\psi(x)$.

The same construction takes place at each step of the resolution process. Assume that a point $x$ of the exceptional divisor of the first resolution has to be blown-up (and thus its image $x'$ under the constructed diffeomorphism as well).

One has two somewhat different situations. The point $x$ (and
thus the point $x'$) is either the intersection point of two
components of the exceptional divisor or a point of only one
component. In the first case one has the representations of the
group $G_x = G_{x'}$ in the tangent spaces to the components. In
the second case one has representations of this group in the
tangent space to the component and in the normal one.
If one fixes local coordinates at the point $x$ in which the
representation of $G_x$ is diagonal, then the diffeomorphism (being
complex analytic by the induction
supposition) defines local
coordinates at the point $x'=\psi(x)$ identifying (complex
analytically) neighbourhoods of $x$ and of $x'$. This gives
complex analytic isomorphisms of  the new born components and of
their tubular neighbourhoods. As above this isomorphism, in
general, does not send intersection points of the new born
components of the first resolution with the strict transforms of
the curves $C_i$ to the corresponding points on the component in
the second resolution. However this can be corrected by a smooth
isotopy which remains complex analytic in neighbourhoods of the
points under consideration.

In this way one gets a diffeomorphism of neighbourhoods of the
exceptional divisors ${\cal D}$
and ${\cal D'}$ of the resolutions $\pi$ and $\pi'$ which sends
the intersection points of the strict transforms of the curves $C_i$
with ${\cal D}$ (and thus of their shifts $gC_i$) to the corresponding points for the curves
$C'_i$. Generally speaking the strict transforms of the curves
$C_i$ do not go to the strict transforms of the curves $C'_i$.
(In fact this may happen only if the isotropy group of the
corresponding point on ${\cal D}$ is trivial.)
This can be corrected by a local diffeomorphism in a
neighbourhood of the intersection point which should be
duplicated at all the points of the $G$-orbit.
\end{proof}

%%%%%%%%%%%%%%%%%%%%%%%%%%%%%%%%%%%%%%%%%%%%%%%%%%%%%
\section{$G$-equivariant Poincar\'e series}\label{defPoin}
%%%%%%%%%%%%%%%%%%%%%%%%%%%%%%%%%%%%%%%%%%%%%%%%%%%%%

In \cite{RMC} the $G$-equivariant Poincar\'e series of a
multi-index filtration defined by a set of valuations or order
functions was defined as an element of the Grothendieck ring of
$G$-sets with an additional structure.

Let $(V,0)$ be a germ of a complex analytic variety with an action
of a finite group $G$. The group $G$ acts on the ring $\OO_{V,0}$ of
germs of functions on $(V,0)$: $g^* f(x)= f(g^{-1}x)$ ($f\in
\OO_{V,0}$, $g\in G$, $x\in V$). A function $v: \OO_{V,0}\to \Z_{\ge
0} \cup \{+\infty\}$ is called an {\it order function} if $v(\lambda
f)= v(f)$ for a non-zero $\lambda\in\C$ and $v(f_1+f_2)\ge \min
\{v(f_1), v(f_2)\}$. (If besides that $v(f_1 f_2)=v(f_1)+v(f_2)$,
the function $v$ is a valuation.) A multi-index filtration of the
ring $\OO_{V,0}$ is defined by a collection $\seq v1r$ of order
functions:
\begin{equation}\label{eq1}
J(\vv) = \{ f\in \OO_{V,0} : \vv(f)\ge \vv\}\; ,
\end{equation}
where $\vv=(\seq v1r)\in \Z^{r}_{\ge 0}$, $\vv(f)= (v_1(f), \ldots,
v_r(f))$ and  $(\seq{v'}1r)\ge (\seq{v''}1r)$ if and only if $v'_i\ge
v''_i$ for all $i=1,\ldots, r$. We assume that the filtration
$J(\vv)$ is {\it finitely determined}, i.e, for any $\vv\in
\Z^{r}_{\ge 0}$, there exists an integer $k$ such that ${\mathfrak
m}^k\subset J(\vv)$ where ${\mathfrak m}$ is the maximal ideal in
$\OO_{V,0}$.

Let $\P\OO_{V,0}$ be the projectivization of the ring $\OO_{V,0}$.
In \cite{IJM} there were defined the notions of cylindric subsets of
$\P\OO_{V,0}$, their Euler characteristics and the integral with
respect to the Euler characteristic over $\P\OO_{V,0}$. In the same way these notions
can be defined for the factor $\P\OO_{V,0}/G$ of $\P\OO_{V,0}$ by the
action of $G$. The (usual) Poincar\'{e} series of the multi-index
filtration can be defined as
$$
P_{\{v_i\}} (\seq t1r) = \int\limits_{ \P \OO_{V,0}} \tt^{\,\vv(f)}
d\chi\; ,
$$
where $\tt=(\seq t1r)$, $\tt^{\,\vv}= t_1^{v_1}\cdot\ldots \cdot
t_r^{v_r}$; $\tt^{\,\vv(f)}$ is considered as a function on $\P
\OO_{V,0}$ with values in the ring (Abelian group) $\Z[[\seq
t1r]]$: see \cite{IJM}.

\begin{definition}
A (``locally finite") $(G,r)\mbox{-set}$ $A$ is a triple $(X, \ww,
\alpha)$ where
\begin{itemize}
\item $X$ is a $G$-set, i.e. a set with a $G$-action;
\item $\ww$ is a function on $X$ with values in
$\Z^r_{\ge 0}$;
\item $\alpha$ associates to each point $x\in X$ a
one-dimensional representation $\alpha_{x}$ of the isotropy group
$G_x= \{ a\in G : ax =x\}$ of the point $x$;
\end{itemize}
satisfying the following conditions:
\begin{enumerate}
\item[1)] $\alpha_{ax}=a \alpha_x a^{-1}$ for $x\in X$,
$a\in G$;
\item[2)] for any $\ww\in \Z^r_{\ge 0}$ the set
$\{x\in X^A: \ww(x)\le \ww\}$ is finite.
\end{enumerate}
\end{definition}

All (locally finite) $(G,r)$-sets form an Abelian semigroup in which
the sum is defined as the disjoint union. The Cartesian product
appropriately defined (see \cite{RMC}) makes the semigroup of
$(G,r)$-sets a semiring. Let $K_0((G,r)\mbox{-sets})$ be the
corresponding Grothendieck ring~--- the Grothendieck ring of
(locally finite) $(G,r)$-sets.

Let the filtration $\{J(\vv)\}$ be defined by the order
functions $\seq v1r$ by (\ref{eq1}). For an element $f\in
\P\OO_{V,0}$, let $G_f$ be the isotropy group of the corresponding point
of $\P\OO_{V,0}$: $G_f =\{a \in G : a^*f =\lambda_f(a)f\}$. The map
$a\mapsto\lambda_f(a)$ defines a one-dimensional representation
$\lambda_f$ of the group $G_f$.
For an element $f\in \P\OO_{V,0}$,
let $T_{f}$ be the element of the Grothendieck ring
$K_0((G,r)\mbox{-sets})$ represented by the orbit $Gf$ of $f$ (as a
$G$-set) with  $\ww^{T_f}(a^*f)=\vv(a^*f)$ and
$\alpha^{T_f}_{a^*f}=\lambda_{a^*f}$ ($a\in G$).

Let us consider $T: f\mapsto T_f$ as a function on $\P \OO_{V,0}/G$
with values in the Grothendieck ring $K_0((G,r)\mbox{-sets})$. This function is cylindric
and integrable (with respect to the Euler characteristic).

\begin{definition} (\cite{RMC}).
The equivariant Poincar\'{e} series $P^G_{\{v_i\}}$ of the
filtration $\{J(\vv)\}$ is defined by
$$
P^G_{\{v_i\}} = \int\limits_{\P\OO_{V,0}/G} T_f d\chi \in
K_0((G,r)\mbox{-sets}) \; .
$$
\end{definition}

Let all the order functions $v_i$ defining the filtration $\{J(\vv)\}$, $i=1, 2, \dots, r$,
be either curve or divisorial valuations. (In general it is possible that some of them
are curve valuations and some of them are divisorial ones, but we shall not consider
this case here.) Let $\pi: ({\cal{X, D}})\to (\C^2,0)$ be a $G$-resolution of this set of valuations: see above.

For a collection of divisorial valuations $\{v_i\}$, let
${\stackrel{\bullet}{E}}_\sigma$ be the ``smooth part" of the
component $E_\sigma$ in ${\cal D}$, i.e. $E_\sigma$ itself minus intersection
points with all other components of the exceptional divisor ${\cal
D}$.
For a collection of curve valuations $\{v_i\}$ corresponding to branches ${C_i}$, let
${\stackrel{\bullet}{E}}_\sigma$ be the ``smooth part" of the
component $E_\sigma$ in the total transform $\pi^{-1}(C)$, i.e. $E_\sigma$ itself minus
intersection points of all the components of  $\pi^{-1}(C)$, $C= \bigcup\limits_{i,g} g C_i$.
Let ${\stackrel{\bullet}{\cal
D}}=\bigcup\limits_\sigma{\stackrel{\bullet}{E}}_\sigma$ and let
$\widehat{\cal D}={\stackrel{\bullet}{\cal D}}/G$ be the
corresponding factor space, i.e. the space of orbits of the action
of the group $G$ on ${\stackrel{\bullet}{\cal D}}$. Let $p:
{\stackrel{\bullet}{\cal D}} \to \h{\cal D}$ be the factorization
map.

For $x\in {\stackrel{\bullet}{\cal D}}$, let the corresponding curvette $L_x$ be given by an
equation $h'_x=0$, $h'_x\in\OO_{\C^2,0}$. Let
$h_x = \sum\limits_{g\in
G_x}{\dfrac{h'_x}{g^*h'_x}}(0)\cdot g^*
h'_x$. The germ $h_x$ is $G_x$-equivariant and $\{h_x=0\}=L_x$.
Moreover, in what follows we assume that the germ $h_x$ is fixed
this way for one point $x$ of each $G$-orbit and is defined by
$h_{gx} = g h_x g^{-1}$ for other points of the orbit.
(This defines $h_{gx}$ modulo a constant factor.)

Let $\{\Xi\}$ be a stratification of the space (in fact of a
smooth curve) $\widehat{\cal D}$ (${\widehat{\cal
D}}=\coprod \Xi$) such that:
\begin{enumerate}
\item[1)] each stratum $\Xi$ is connected;
\item[2)] for each point $\h{x}\in \Xi$ and for each point $x$
from
its preimage $p^{-1}(\h{x})$, the conjugacy class of the
isotropy
subgroup $G_x$ of the point $x$ is the same, i.e. depends only on
the stratum $\Xi$.
\end{enumerate}
The last is equivalent to say that, over each
stratum $\Xi$, the map
$p: {\stackrel{\bullet}{\cal D}} \to \h{\cal D}$ is a covering.

For a component $E_{\sigma}$ of the exceptional divisor
${\cal D}$, let $v_{\sigma}$ be the corresponding divisorial
valuation on the ring $\OO_{\C^2,0}$: for $f\in \OO_{\C^2,0}$,
$v_\sigma(f)$ is the order of zero of the lifting $f\circ\pi$ of
the function $f$ along the component $E_\sigma$. Let
$\{\sigma_1,\ldots, \sigma_r\}$ be a subset of $\Gamma$, and let $\seq v1r$ be
the corresponding divisorial valuations. They define the
multi-index filtration (\ref{eq1}).

For a point $x\in {\stackrel{\bullet}{\cal D}}$, let $T_x$ be the
element of the Grothendieck ring $K_0((G,r)\linebreak[0]\mbox{-sets})$ defined
by $T_x = T_{h_x}$ where $h_x$ is a $G_x$-equivariant function
defining a curvette at the point $x$.
The element $T_x$ is well-defined. i.e. does not depend on the
choice of the function $h_x$. One can see that the
element $T_x$ is one and the same for all points from the
preimage of a stratum $\Xi$ and therefore it will be denoted by
$T_{\Xi}$.

\begin{theorem} {\rm \cite{RMC}}
\begin{equation}\label{eq2}
P^G_{\{v_i\}} = \prod\limits_{\{\Xi\}} (1-T_{\Xi})^{-\chi(\Xi)}
\; .
\end{equation}
\end{theorem}

\begin{statement}
The initial action (representation) of the group $G$ on $\C^2$, the $G$-resolution
graph of a collection of curve or divisorial valuations plus the correspondence
between the tails of the graph and the special points on $E_1$ determine the
$G$-Poincar\'e series of the set of valuations.
\end{statement}

\begin{proof}
In order to compute the $G$-Poincar\'e series $P^G_{\{v_i\}}$
using Equation~(\ref{eq2}) one has to describe the stratification $\{\Xi\}$
and to
determine, for each stratum $\Xi$, its Euler characteristic, the
corresponding isotropy subgroup $G_{\Xi}$ (the $G$-set
representing $T_{\Xi}$ is just $G/G_{\Xi}$), the function
$\ww_{\Xi}$ and the one-dimensional representation $\alpha_x$ of
the isotropy subgroup $G_x$, $x\in p^{-1}(\Xi)$.
Theorem 1 implies the
the stratification and the corresponding isotropy subgroups are
determined by the described data.
The function $\ww_\Xi$, i.e. the
values for the corresponding curvettes, is computed from the
resolution graph in the standard way (it depends only on the
corresponding component of the exceptional divisor). Thus the only
remaining problem is to determine the representation $\alpha_x$ for $x\in \Xi$.
Let $u$ and $v$ be local $G_x$-equivariant coordinates in a
neighbourhood of the point $x$ such that the corresponding
component of the exceptional divisor is given by the equation
$u=0$. By Theorem 1, the action of $G_x$ on the function $u$ (or rather the dual
action in the normal space) is determined by the described data.
Let $\omega = dx\wedge dy$ be a $G$-equivariant 2-form
form on $\C^2$. The representation of $G$ on $\C^2$ determines
the action of $G$ on (the one dimensional space generated by)
$\omega$. One has
$(\pi^* \omega)_x = \varphi(u,v) u^{\nu} du\wedge dv$, where
$\varphi(0,0)\neq 0$ and the multiplicity $\nu$ is determined by
the resolution graph (see, e.g., \cite[Section 8.3]{Wall}. From the action of $G_x$ on $\omega$ and on
$u$ one gets the action on $v$. If $h_x=0$ is a $G$-equivariant
equation of the curvette $\pi(\{v=0\})$  at the point $x$ one has
$\pi^* h_x = \psi(u,v)u^{m} v$, where $\psi(0,0)\neq 0$ and $m$
can be computed from the resolution graph. Thus the action
$\alpha_x$ of $G_x$ on $h_x$ is also determined.
\end{proof}

In what follows we shall use the following property of the the Grothendieck ring
$K_0((G,r)\mbox{-sets})$.
The ring $K_0((G,r)\mbox{-sets})$ has the maximal ideal $\mathfrak{M}$ which is
generated by all irreducible $(G,r)$-sets different from $1$. Let an element $P$
belong to $1+\mathfrak{M}$. Then it has a unique representation in the ``A'Campo
type form": $P=\prod(1-T)^{s_T}$, where $T$ runs over all irreducible elements of
$K_0((G,r)\mbox{-sets})$ and $s_T$ are integers.
In general this product contains infinitely many factors. Theorem 2 implies that if
$P$ is the $G$-Poincar\'e series of a set of curve or divisorial valuations on $\OO_{\C^2,0}$, then this
product is finite.

%%%%%%%%%%%%%%%%%%%%%%%%%%%%%%%%%%%%%%%%%%%%%%%%%%%%%
\section{$G$-Poincar\'e series and $G$-topology of
sets of divisors}\label{sec3}
%%%%%%%%%%%%%%%%%%%%%%%%%%%%%%%%%%%%%%%%%%%%%%%%%%%%%

Let $(\C^2,0)$ be endowed by a $G$-action and let $\{v_i\}$,
$i=1,\ldots,r$, be a set of divisorial valuations on
$\OO_{\C^2,0}$.

\begin{theorem}\label{div_top}
The $G$-equivariant Poincar\'e series $P^G_{\{v_i\}}$ of the set
of the divisorial valuations $\{v_i\}$ determines the
$G$-equivariant topology of the set of divisorial valuations.
\end{theorem}

\begin{proof}
Let $v_{ig}(\varphi):= v_i((g^{-1})^* \varphi)$ be the
valuation defined by the shift $gE_i$ of the component $E_i$. Let
$\{v_{ig}\}$ be the corresponding set of valuations on the ring
$\OO_{\C^2,0}$ numbered by the set $\{1,\ldots,r\}\times G$. One
can easily see that the $G$-Poincar\'e series $P^G_{\{v_i\}}$ of
the set of valuations $\{v_i\}$ determines the $G$-Poincar\'e
series $P^{G}_{\{v_{ig}\}}$ of the set $\{v_{ig}\}$.
Indeed,
$P^{G}_{\{v_{ig}\}}$ is represented by the same $G$-set $X$
as
$P^G_{\{v_i\}}$ with the same function $\alpha$ and with
$\ww: X\to \Z_{\ge 0}^{r\vert G\vert}$ defined by
$w_{ig}(x)=w_i((g^{-1})^* x)$, $x\in X$.
As it was explained in \cite[Statement 2]{RMC} the
$G$-Poincar\'e series $P^G_{\{v_{ig}\}}$ determines the usual
(non-equivariant) Poincar\'e series $P_{\{v_{ig}\}}(\tt)$.
The (usual) Poincar\'e series
$P_{\{v_{ig}\}}(\tt)$ determines the minimal resolution graph of
the set of valuations $\{v_{ig}\}$: \cite{FAOM}.
(Formally speaking, in \cite{FAOM} it was assumed that the
divisorial valuations in the set are different. However one can easily see
that there is no difference if one permits repeated valuations.)
Moreover the action of the group $G$ on the set $\{v_{ig}\}$ of
valuations induces a $G$-action on the minimal resolution graph.
By Theorem 1 one has to show that the $G$-Poincar\'e series
$P^G_{\{v_{ig}\}}$ determines the representation of $G$ on $\C^2$
and the correspondence between ``tails" of the resolution graph
emerging from the special points on the first component of the
exceptional divisor and these points. (If there are no special
points on the first component of the exceptional divisor (this
can happen only if
$G$ is cyclic), only the representation of $G$ on $\C^2$ has to
be determined.)

Let us consider the case of an Abelian group $G$
first. If there are no special points on the first component
$E_{\bf 1}$ of the exceptional divisor, all points of $E_{\bf 1}$
are fixed with respect to the group $G$, the group $G$ is cyclic
and the representation is a scalar one. This (one dimensional)
representation is dual to the representation of the group $G$ on
the one-dimensional space generated by any linear function. The
case when there are no more components in ${\cal D}$, i.e. if the
resolution is achieved by the first blow-up, is trivial.
Otherwise let us consider a maximal component $E_{\sigma}$ among those components $E_{\tau}$ of the exceptional divisor for which $G_{\tau}=G$ and the corresponding curvette is smooth. (The last condition can be easily detected from the resolution graph.)
The smooth part $\stackrel{\bullet}{E}_{\sigma}$ of this component contains a special point with $G_x=G$ (or all the points of $\stackrel{\bullet}{E}_{\sigma}$ are such that $G_x=G$). The point(s) from $\stackrel{\bullet}{E}_{\sigma}$ with $G_x=G$ bring(s) into the equation~\ref{eq2} a factor of the form $(1-T_{\Xi})^{-1}$ where
$T_{\Xi}$ is represented by the $G$-set consisting of one point and which cannot be
eliminated by other factors. The ($G$-equivariant) curvette $L$ at the described
special point of the divisor is smooth. Therefore the representation of $G$ on the
one-dimensional space generated by a $G$-equivariant equation of $L$ coincides with
the representation on the space generated by a linear function.

Let us take all factors in the representation of the Poincar\'e series
$P^G_{\{v_i\}}$
of the form $\prod_{T}(1-T)^{s{_T}}$, where $T$ is represented by the $G$-set
consisting of one point and with $s_T=-1$. For each of them, the corresponding
$\ww$ of $T$ determines the corresponding component of the exceptional divisor and
therefore the topological type of the corresponding curvettes. Therefore one can
choose a factor which corresponds to a component with a smooth curvette and the
representation $\alpha^{T}$ gives us the representation on the space generated by a
linear function.

Let the first component $E_{\bf 1}$ contain two special points. Without loss of
generality we can assume that they correspond to the coordinate axis $\{x=0\}$ and
$\{y=0\}$. The representation of the group $G$ on $\C^2$ is defined by its action on
the linear functions $x$ and $y$. For each of them this action can be recovered from
a factor of the form described above just in the same way.
Moreover, a factor, which determines the action of the group $G$ on the function $x$,
corresponds to a component of the exceptional divisor from the tail emerging from the point $\{x=0\}$.

Now let $G$ be an arbitrary (not necessarily Abelian) group. For an element $g\in
G$ consider the action of the cyclic group $\langle g\rangle$ generated by $g$ on
$\C^2$. One can see that the $G$-equivariant Poincar\'e series $P^{G}_{\{v_i\}}$
determines the $\langle g\rangle$-Poincar\'e series
$P^{\langle g\rangle}_{\{v_i\}}$ just like in \cite[Proposition 2]{RMC}. This
implies that the $G$-equivariant Poincar\'e series determines the representation of the
subgroup $\langle g\rangle$. (Another way is to repeat the arguments above adjusting
them to the subgroup $\langle g\rangle$.) Therefore the $G$-Poincar\'e series
$P^{G}_{\{v_i\}}$ determines the value of the character
of the $G$-representation on $\C^2$ for each element $g\in G$ and thus the
representation itself.

Special points of the $G$-action on the first component $E_{\bf 1}$ of the
exceptional divisor correspond to some Abelian subgroups $H$ of $G$. For each such
subgroup $H$ there are two special points corresponding to different one-dimensional
representations of $H$. Again the construction above for an Abelian group permits to
identify tails of the dual resolution graph with these two points. This finishes the
proof.
\end{proof}

%%%%%%%%%%%%%%%%%%%%%%%%%%%%%%%%%%%%%%%%%%%%%%%%%%%%%
\section{$G$-Poincar\'e series and $G$-topology of
curves}\label{sec4}
%%%%%%%%%%%%%%%%%%%%%%%%%%%%%%%%%%%%%%%%%%%%%%%%%%%%%
A connection between the equivariant Poincar\'e series of a set of curve valuations and the equivariant topology of the corresponding curve singularity is more involved than for divisorial valuations.

\begin{example}
One can see that the collections $\{C_i\}$ and $\{C'_i\}$ from Example 1 or 2 in
Section~1 have the same $G$-Poincar\'e series. Namely
$P^{G}_{\{C_i\}}= P^{G}_{\{C'_i\}}= (1-T)$, where $T$ is the $(G,3)$-set defined by
$X=G/(e)$ ($G=\Z_{15}$ or $\Z_7$ in Examples 1 and 2 respectively), $w$ is the
constant function on $X$ with $w(x)=(2,1,2)$. (The representation
$\alpha(x)$ is trivial being a representation of the trivial group $(e)$).
Thus the $G$-Poincar\'e series does not, in general, determines the topology of the
set of curves.
\end{example}

\begin{remark}
One can see that the $G$-Poincar\'e series of the set of divisorial valuations from
Example 3 are different since the factors corresponding to the special points on the
second divisors are different. This factors do not appear in the $G$-Poincar\'e
series of curves in Example 1 since the strict transforms of branches pass through these
points.

We shall show that the effect like the one in the Example occurs only if, among the
branches of the curve, there are smooth branches invariant with respect to a
non-trivial element $g$ of the group $G$ whose action on $\C^2$ is not a scalar one
(i.e. the representation  of the cyclic subgroup $\langle g\rangle$ generated by $g$ is the sum
of two different one-dimensional representations).
\end{remark}

\begin{theorem}\label{curve_top}
Let $\C^2$ be equipped with a faithful action of a finite group $G$ and let
$\{C_i\}$, $i=1, \ldots, r$, be a collection of irreducible curve singularities in
$(\C^2, 0)$ such that it does not contain curves from the same $G$-orbit and it does not contain a smooth curve invariant with respect to a non-trivial element of $G$ whose action on $\C^2$ is not a scalar one.
Let $\{v_i\}$ be the corresponding collection of valuations. Then the $G$-equivariant Poincar\'e series $P^G_{\{v_i\}}$ of the collection
$\{v_i\}$ determines the minimal
$G$-resolution graph of the curve $\bigcup_{i=1}^r C_i$ and
%% therefore
the $G$-equivariant topology of the pair $(\C^2, \bigcup_{i=1}^r C_i)$.
\end{theorem}

\begin{proof}
Let $\{v_{ig}\}$ ($i=1,\ldots, r; g\in G$), where
$v_{ig}(\varphi) = v_i((g^{-1})^{*}\varphi)$ is the set of
``$G$-shifts" of the valuations $v_i$ (it is possible that
$v_{i_1g_1}=v_{i_2g_2}$ and even that $v_{ig_1}=v_{ig_2}$). The
$G$-Poincar\'e series
$P^G_{\{v_{i}\}}$ determines the
$G$-Poincar\'e series
$P^G_{\{v_{ig}\}}$ just in the same way as in the divisorial case in Theorem~\ref{div_top}.

Since, with every valuation from the collection $\{v_{ig}\}$,
this collection contains also its $G$-shifts, the remark after
Statement 2 from \cite{RMC} implies that the usual (non-equivariant) Poincar\'e series
$P_{\{v_{ig}\}}(\{t_{ig}\})$ is determined by the $G$-Poincar\'e series $P^G_{\{v_{i}\}}$.

The collection $\{v_{ig}\}$ contains repeated (curve) valuations.
Therefore we need a (somewhat more precise) version of Theorem 2
from \cite{RMC} for collections of curve valuations with
(possibly) repeated ones. Assume that $\{C_i\}$ and $\{C'_i\}$,
$i=1,\ldots, r$, are collections of branches in $(\C^2,0)$
possibly with $C_{i_1}=C_{i_2}$ for $i_1\neq i_2$ (and/or
$C'_{j_1}=C'_{j_2}$ for $j_1\neq j_2$) and let $\{v_i\}$ and
$\{v'_i\}$ be the collections of the corresponding valuations.
Essentially the same proof as in \cite[Theorem 2]{RMC} gives the
following proposition.

\begin{statement}
If $P_{\{v_i\}}(\seq t1r) = P_{\{v'_i\}}(\seq t1r)$ then there
exists a
combinatorial equivalence of the minimal resolution graphs of the
reductions of the curves $\bigcup_{i=1}^rC_i$
and $\bigcup_{i=1}^rC'_i$, i.e. components of the exceptional
divisors corresponding to equivalent vertices intersect the
same numbers of the strict transforms of different curves from the collections $\{C_i\}$
and $\{C'_i\}$.  (The distribution of these curves into the
groups of equal ones can be different).
\end{statement}

The  knowledge of the usual Poincar\'e series
of the collection $\{v_{ig}\}$ gives
the minimal resolution graph (the usual, not the $G$-equivariant
one) of the collection. Moreover the action of
the group $G$ on the collection $\{v_{ig}\}$
of valuations determines the $G$-action on the resolution graph.

Like in Theorem 2 one has to show that the $G$-Poincar\'e series
$P^{G}_{\{v_i\}}$ determines the representation of $G$ on $\C^2$
and
the correspondence between the ``tails" of the resolution graph.
This is made just in the same way as in Theorem 2 for divisorial
valuations since there are no strict transforms of the branches
$gC_i$ at points of the exceptional divisor with $G_x = G$ and
the corresponding curvette smooth and therefore the corresponding
factor
$(1- T _{\Xi})^{-1}$ (which permits to restore the action of $G$
on the corresponding linear function: a coordinate) is contained
in the A'Campo type decomposition of the Poincar\'e series.
\end{proof}

\begin{remark}
One can see that the $G$-Poincar\'e series of a collection of curve valuations determines whether or not
the collection of curves satisfies the conditions of Theorem~\ref{curve_top}.
\end{remark}

Adresses:

Universidad de Valladolid, Instituto de Investigaci\'on en Matem\'aticas (IMUVA), 47011 Valladolid, Spain.
\newline E-mail: campillo\symbol{'100}agt.uva.es, fdelgado\symbol{'100}agt.uva.es

Moscow State University, Faculty of Mathematics and Mechanics, Moscow, GSP-1, 119991, Russia.
\newline E-mail: sabir\symbol{'100}mccme.ru

\end{document}